\documentclass[a4paper, 10pt, reqno,final]{article}

\usepackage[T1]{fontenc}		     
\usepackage[utf8]{inputenc}			 
\usepackage[english]{babel}

\usepackage{amsmath}
\usepackage{amssymb}
\usepackage{amsthm}
	\theoremstyle{plain}

		\newtheorem{thm}{Theorem}[section]	
		\newtheorem{cor}[thm]{Corollary}	
		\newtheorem{lem}[thm]{Lemma}		
		
	\theoremstyle{definition}
		\newtheorem{defn}[thm]{Definition}	
		\newtheorem{ex}[thm]{Example}		

	\theoremstyle{remark}
		\newtheorem{rem}[thm]{Remark}		

\numberwithin{equation}{section}	

\usepackage{mathtools}		
\usepackage{mathrsfs}		
\usepackage{eucal}			

\usepackage{braket}			
\usepackage[all,pdf]{xy}		
\usepackage{mparhack}		
\usepackage{relsize}			

\usepackage{a4wide}		

\usepackage{booktabs}		
\usepackage{multirow}		
\usepackage{caption}		
\usepackage{rotating}
\usepackage{subfig}

\usepackage{varioref}		
\usepackage{footmisc}

\usepackage{graphicx}		
\usepackage{epsfig}

\usepackage{enumerate}		
\usepackage[colorlinks=true, linkcolor=black, citecolor=black,
urlcolor=blue]{hyperref}		
\mathtoolsset{showonlyrefs}

\newcommand{\R}{\mathbb{R}}	

\newcommand{\C}{\mathbb{C}}		
\newcommand{\Q}{\mathbb{Q}}		
\newcommand{\T}{\mathbb{T}}		
\newcommand{\U}{\mathbb{U}}		

\newcommand{\Sp}{\mathrm{Sp}}

\newcommand{\Graph}{\mathrm{Gr\,}}

\newcommand{\Sym}{\mathrm{Sym}}

\newcommand{\Gr}{\mathrm{Gr}}

\newcommand{\N}{\mathbb{N}}		

\newcommand{\MasSymp}{\iota}
\newcommand{\mIndex}{\mathcal{I}}

\newcommand{\Z}{\mathbb{Z}}		

\newcommand{\noo}[1]{\overset {\mbox{%
\lower1pt\hbox{${\scriptscriptstyle o}$}}}n^{\mbox{%
\lower2pt\hbox{$\scriptscriptstyle #1$}}}}

\renewcommand{\leq}{\leqslant}
\renewcommand{\geq}{\geqslant}
\renewcommand{\hat}{\widehat}
\renewcommand{\tilde}{\widetilde}
\renewcommand{\=}{\coloneqq}			

\newcommand{\email}[1]{\href{mailto:#1}{\textsf{#1}}}

\usepackage{bm}

\newcommand{\Id}{I}

\title{Mean index for non-periodic orbits in  Hamiltonian systems }
\author{Xijun Hu \thanks{Partially supported by NSFC (No. 11790271, 11425105) },
Li Wu \thanks{Corresponding author, partially supported by NSFC (No. 11425105) }, }
\date{\today}

\date{\today}
\begin{document}
 \maketitle

\begin{abstract} 
In this paper, we define mean index for non-periodic orbits in Hamiltonian systems and 
study its properties. In general, the mean index is an interval in $\R$  which is  uniformly continuous on the systems. We show that the index interval is  a  point for a quasi-periodic orbit.  The mean index can be considered as a generalization of  rotation number which defined by  Johnson and Moser  in the study of 
almost periodic Schrodinger operators.  Motivated by their works, we study the relation of Fredholm property of the linear operator and the mean index at the end of the paper.

\end{abstract}


\section{Introductions }\label{sec:definition}
In this paper, we consider the following linear Hamiltonian system
\begin{equation}\label{eq:general linear_system}
\dot z=JB(t)z, \quad  t\in  \R ,
\end{equation}
where $J\= \begin{bmatrix}
	0 & -\Id\\ \Id & 0
	\end{bmatrix}$ denotes the {\em standard symplectic matrix\/}, $B(t)\in \Sym(d,\R)$  the set of all $d \times d$ symmetric matrices. Throughout of the paper, we assume 
\begin{itemize}
	 \item[{\bf(L1)\/}]  $B\in E_K:=\{B \in C^0(\R, \Sym(2d,\R))| \|B\| \le K\}$ for some $K>0$.	 
	\end{itemize}
Let $\gamma(t)$ be the fundamental solution matrix of \eqref{eq:general linear_system}, that is  $\dot{\gamma}(t)=JB(t)\gamma(t), \gamma(0)=\Id$.  It is well known that $\gamma(t)\in \Sp(2d):=\{M\in GL(\R^{2d})| M^TJM=J\}$.  In the periodic case, that is  $B(t)=B(t+T)$, the Maslov-type index $$i_\omega(\gamma, [0,T])\in\Z, \, \omega\in\U$$ is well defined, and the mean index per period is defined by 
\begin{equation}\label{eq:mean index period}
\hat{i}(\gamma)=\lim_{k\to \infty}\frac{i_1(\gamma, [0,kT])}{k}.
\end{equation}
It is   an important tool  in study the multiplicity and 
stability of periodic orbits in Hamiltonian systems \cite{Eke90, Lon02}.

In the case of non-periodic,  Ekeland \cite{Eke90}   defined  the mean index $\mIndex$ by the limit of   $\frac{i_1(\gamma, [0,T])}{T}$, but as pointed out by him,  no reason for  $\frac{i_1(\gamma, [0,T])}{T}$ 
converges  unless $B(t)$ happens to be periodic. Ekeland proved that, $\mIndex$
 exists  almost every where on  the energy level of convex Hamiltonian system. 
  There are few results  about the mean index of non-periodic trajectory.  Only recently, Zhou, Wu and Zhu \cite{ZWZ18} gave a generalization of   Ekeland's  almost existence  theorem.  In general, we  almost know nothing about the mean index. Motivated by their works, we give the following definition.

 \begin{defn}\label{def:UL_index}
 We define the positive upper mean index   $\mIndex_U(\gamma)$ and lower mean index $\mIndex_L(\gamma)$ of the linear system
	 $\eqref{eq:general linear_system}$ as  follows:

\begin{equation}
	 \begin{cases}
	 \mIndex_U(\gamma)=	\varlimsup_{l\to +\infty} \frac{i_1(\gamma(t),[0,l])}{l} \\
	 	\mIndex_L(\gamma)=\varliminf_{l\to +\infty} \frac{i_1(\gamma(t),[0,l])}{l}
	\end{cases}.
\end{equation}  
Similary, we define
\begin{equation}
	 \begin{cases}
	 \mIndex^-_U(\gamma)=	\varlimsup_{l\to +\infty} \frac{i_1(\gamma(t),[0,-l])}{-l} \\
	 	\mIndex^-_L(\gamma)=\varliminf_{l\to +\infty} \frac{i_1(\gamma(t),[0,-l])}{-l}
	\end{cases},
\end{equation}  
where $\gamma(t),[0,-l]$  is the path $\gamma(-t), t\in[0,l]$.

 \end{defn}
Since the system \eqref{eq:general linear_system} is determined by $B(t)$, we can also denote the upper and lower mean index by $\mIndex_U(B)$ and $\mIndex_L(B)$ respectively. 

We  give Example \ref{ex:non pt} to show that it is possible $ \mIndex_U(\gamma)\neq  \mIndex_L(\gamma)$.  In this case, we proved  that for any $\alpha\in  [\mIndex_L(\gamma),  \mIndex_U(\gamma)]$, there exists a sequence $t_n\to \infty$, such that 
\begin{equation} \label{eq:lim betw}\lim_{n\to\infty}\frac{i_1(\gamma,[0,t_n])}{t_n}=\alpha. \end{equation}
Please refer to Lemma \ref{lem:index_interval} for the detail.
Then we define positive mean index set and negative mean index set
by 
\begin{equation} \label{eq:index set} \mIndex(B)=[\mIndex_L(\gamma),  \mIndex_U(\gamma)],\quad \mIndex^-(B)=[\mIndex^-_L(\gamma),  \mIndex^-_U(\gamma)]. \end{equation}
Obviously, in the $T$-periodic case,  both $\mIndex(B)$ and $\mIndex^-(B)$ are points and satisfy
\begin{equation} \label{eq:index set per} \hat{i}(\gamma,[0,T]) =T \mIndex(B)=T \mIndex^-(B). \end{equation}

Please note that $\mIndex(B), \mIndex^-(B)$ are invariant under translation of time, and they do not depend on the value of $B(t)$ at any finite interval. Roughly speaking, they are only depend on the value of $B(t)$ at infinity. Please refer to Corollary \ref{cor:trans_invariant} for the detail. Moreover, we proved that $\mIndex(B)$ and $\mIndex^-(B)$ are uniformly  continuous on $E_K$. 

\begin{thm}\label{thm:uniform conti}  $\mIndex_U, \mIndex_L, \mIndex^-_U, \mIndex^-_L$ are uniformly continuous on $E_K$.
\end{thm}

Let $x$ be  a orbit  of a $C^2$ Hamiltonian systems, that is $\dot{x}=J\mathcal{H}'(x(t))$, $\mathcal{H}\in C^2(\R^{2n}, \R)$. Let $B(t)=\mathcal{H}''(x(t))$, we define 
$$\mIndex(x)=\mIndex(B), \quad \mIndex^-(x)=\mIndex^-(B),$$
 and $\mIndex(\xi)=\mIndex(x)$, where $x(0)=\xi$.  It is obvious that $\mIndex, \mIndex^-$ are constant along the orbit.
  Assume $\Lambda$ is an invariant set of the Hamiltonian flow. We define 
\begin{equation} \label{def:index inv set} \mIndex(\Lambda)=\cup_{\xi\in\Lambda} \mIndex(\xi), \quad \mIndex^-(\Lambda)=\cup_{\xi\in\Lambda} \mIndex^-(\xi). \end{equation}

An invariant set  is uniquely ergodic if there is precisely one invariant probability measure with the Hamiltonian flow.  A special case is the quasi-periodic orbit.  More precisely, let $\Lambda(x):=\overline{\{x(t), t\in\R\}}$ which is diffeomorphic to torus $\T^n$ and let $D:  \T^n\to \Lambda(x)$ be the homeomorphism, then  $x(t)=D(D^{-1}x(0)+\omega t)$  with $\omega=(\omega_1,\cdots,\omega_n)$ which are independent over $\Q$. In this case, $\Lambda(x)$ is uniquely ergodic.  The following theorem shows that the mean index for a quasi-periodic orbit is a point.

\begin{thm}\label{thm:quasi index pt} For a quasi-periodic orbit $x$, then $\lim_{T\to\infty}\frac{i_1(\gamma,[0,T])}{T} $ exists  and  $$\mIndex(\Lambda_x)=\mIndex^-(\Lambda_x)=\lim_{T\to\infty}\frac{i_1(\gamma,[0,T])}{T}.$$
\end{thm}

For a bounded orbit $x$,  the $\omega$-limit set ($\alpha$-limit set)of $x$  is denoted by $\Lambda_\omega(x)$ ($\Lambda_\alpha(x)$).  $\Lambda_\omega(x)$ ($\Lambda_\alpha(x)$) is a compact invariant set.  Obviously, for a quasi-periodic orbit $x$, $\Lambda_\omega(x)=\Lambda_\alpha(x)=\Lambda(x) $.

\begin{cor}\label{cor:limit set} Assume $\Lambda_\omega(x)=\Lambda(\tilde{x})$ for some quasi periodic orbit $\tilde{x}$, then    we have 
\begin{equation} \label{eq: limit include} \mIndex(x)=\mIndex(\Lambda_\omega(x))=\mIndex(\Lambda(\tilde{x})).\end{equation}
Same result holds for $\Lambda_\alpha(x)$.
\end{cor}

We say an orbit $x$ is heteroclinic  to quasi-periodic if  the $\omega$-limit  set and $\alpha$-limit set are invariant torus of  some quasi-periodic orbits respectively.
	This theorem shows that for a heteroclinic orbit $x$ to quasi-periodic orbits,   then $\mIndex(x), \mIndex^-(x)$ is a point.

In \cite{JM82}, Johnson and Moser define a rotation number for almost periodic Schrodinger operator $\mathcal{L}=\frac{-d^2}{dx^2}+q(x)$, and use this rotation number to discuss the spectrum of $\mathcal{L}$.  In fact, the mean index can be thought as a generalization of rotation number in higher dimension,   and it  will be   explained in Section  3.

Motivated by Johnson and Moser's works, we use  mean index  to study the  Fredholm property of  the following operator
\begin{equation} \label{eq: operator} \mathcal{A}:-J\frac{d}{dt}-B, \quad on \quad L^2 (\R, \R^{2n}),\end{equation}
where  $(Bx)(t)=B(t)x(t)$.  We say $B$ is  asymptotic periodic if there exists periodic $\tilde{B}\in E_K$ ( $B(t+T)=B(t)$), such that $$\lim_{t\to\pm\infty}\|B(t)-\tilde{B}(t)\|=0.$$
We proved that
\begin{thm}\label{thm:fredholm}  Assume $B$ is asymptotic periodic, then  $\mathcal{A}$ is Fredhom if and only if  $\mIndex(B+\lambda)$, $\mIndex^-(B+\lambda)$ is independent of $\lambda$ for $|\lambda|$ small enough. 
\end{thm}

This paper is organized as follows. In Section 2, we prove some basic properties of the mean index. In Section 3, we study the quasi-periodic orbits and study the relation with Fredholm property at Section 4. We give an appendix for Maslov-type index at Section 5.

\section{The property of Lower and upper index}\label{sec:property_mean_index}

In this section, we will prove some fundamental properties  of  upper and lower mean index.
 We will only consider $\mIndex_U(B),\mIndex_L(B)$, since every property will also hold for $\mIndex^-_U(B),\mIndex^-_L(B)$ if we change $B(t)$ with $B(-t)$.

In Definition \ref{def:UL_index}, we use $i_1$ to define mean index. In fact, it can be defined by  Maslov-type index $\MasSymp(M,\gamma)$ with any $M\in\Sp(2n)$.    Please refer to Section 5 for the detail of Maslov index. 

\begin{lem}\label{lem:boundary_invariant}
For any $M\in\Sp(2n)$, 		We have 
	\begin{align*}
		\mIndex_L(B)=\varliminf\MasSymp(M,  \gamma(t),[0,l])/l \\
		\mIndex_U(B)=\varlimsup\MasSymp(M,  \gamma(t),[0,l])/l .
	\end{align*}
\end{lem}
\begin{proof}   This is from the fact that $i_1(\gamma)+d= \MasSymp(I,\gamma)$ and the comparison theorem  \ref{thm:comparison}.
Please refer  to    \eqref{index fini d}. \end{proof}

\begin{lem}\label{lem:monotone}
	The upper and lower mean index is monotone for $B$.
	If $B_1(t)\ge B_2(t), t\in \R$, then  $\mIndex_U(B_1)\ge \mIndex_U(B_2)$ and $\mIndex_L(B_1)\ge \mIndex_L(B_2)$ 
\end{lem}
\begin{proof}
	It is a direct consequence of the monotone property of Maslov-type index.
\end{proof}

Use this lemma, we see that   $\mIndex(K\Id) \ge \mIndex_{U}(B),\mIndex_{L}(B) \ge \mIndex(-K \Id)$.
We can calculate $\mIndex(\pm K\Id)$ directly, and  we get the  bound of  upper and lower mean index.
\begin{lem}
	$d K/\pi \ge \mIndex_U(B)$ and  $\mIndex_L(B)\ge -dK/\pi $.
\end{lem}
\begin{proof}
	Note that     $2d[K l/(2\pi)] \le i_1(e^{KJt},[0,l])\le 2d[K l/(2\pi)]+2d$. 
	We have
	\[
	\mIndex(K\Id)=\lim_{l\to +\infty} \frac{i_1(e^{KJt},[0,l])}{l}=\lim_{l\to +\infty} \frac{2d [K l/2\pi]}{l} =d K/\pi .
	\]	
	Similarly $\mIndex(-KI)=-dK/\pi$.
	Then the lemma follows.
\end{proof}

We need some 
 estimate of Maslov index to prove other properties.
\begin{lem} \label{lem:maslov_estimate}
	$|\MasSymp(I, \gamma,[a,b])|\le d K(b-a)/\pi +4d $.
\end{lem} 
\begin{proof}
	Note that  
	\[\MasSymp(\Id,\gamma(t),t\in [a,b])=\MasSymp(\gamma(a)^{-1},\gamma(t)\gamma(a)^{-1},t\in [a,b]).\]
		\[
	0\le \MasSymp(\Id,\gamma(t)\gamma(a)^{-1},t\in [a,b])-\MasSymp(\gamma(a)^{-1},\gamma(t)\gamma(a)^{-1}),t\in [a,b])\le 2d.
	\]
We have 
	\[
	-2d [K(b-a)/(2\pi)] \le \MasSymp( \Id,\gamma(t)\gamma(a)^{-1}, t\in [a,b])\le 2d[K(b-a)/(2\pi)]+2d.
	\]
	The lemma then follows.
\end{proof}
Now we can prove that the upper and lower mean index are invariant under translation.
\begin{cor}\label{cor:trans_invariant}
	Let $B_s(t)=B(s+t)$, we have $\mIndex_U(B)=\mIndex_U(B_s)$ and $\mIndex_L(B)=\mIndex_L(B_s)$.
\end{cor}
\begin{proof}
	The fundamental solution related to $B_s$ is $ \gamma(s+t)\gamma(s)^{-1}$.
	By Lemma \ref{lem:boundary_invariant}, we have 
	\[
	\mIndex_U(B_s)=\varlimsup_{l\to +\infty} \MasSymp(\gamma(s), \gamma(s+t),t\in[0,l])/l=\varlimsup_{l\to +\infty} i_1( \gamma(s+t),t\in [0,l])/l
	\]	
	By Lemma  \ref{lem:maslov_estimate}, we have
	\[
	|i_1(\gamma(t),t\in [0,l])-i_1(\gamma(s+t),t\in [0,l])|=|i_1(\gamma(t),[0,s])-i_1(\gamma(t),[l,l+s])|\le  2dKs/\pi +4d.
	\]	
	It follows that 
	\[
	\mIndex_U(B)=\varlimsup i_1(\gamma(t),[0,l])/l=\varlimsup i_1(\gamma(s+t),t\in [0,l])/l=\mIndex_U(B_s).
	\]
	Similarly we have $\mIndex_L(B)=\mIndex_U(B_s)$.
\end{proof}
\begin{lem} \label{lem:index_interval}
	
	For each  $v \in [\mIndex_L(B),\mIndex_U(B)]$ and $u>0$, there is  a series of integer $m_k \to +\infty$ such that 
	\[
	\lim_{k\to +\infty} i_1(\gamma,[0, u m_k])/u m_k =v.
	\]
	
\end{lem}
\begin{proof}
	We only need to prove the lemma for $u=1$.
	Consider the series  $a_k =i_1(\gamma,[0,k])/k$	.
	For each $l\in [k,k+1]$, we have 
	\[
	i_1(\gamma(t),[0,l])/l-a_{k}=\frac{i_1(\gamma,[k,l])}{l} - \frac{i_1(\gamma,[0,k])}{kl}.
	\]
	By Lemma \ref{lem:maslov_estimate},  we have
	\begin{equation}\label{eq:ind_dif}
		|i_1(\gamma(t),[0,l])/l-a_k|\le  \frac{nK/\pi +2d}{l}+ \frac{nkK/\pi +2d}{kl}\le  \frac{nK/\pi +2d}{k}+ \frac{d kK/\pi +2d}{k^2}.
	\end{equation}
	
	It follows that there is a series of integers $p_k $ such that 
	\[
	\varlimsup_{l\to +\infty} i_1(\gamma(t),[0,l])/l =\lim_{k\to +\infty} a_{p_k}.
	\]
	Similarly there is a series of integers $q_k$   such that 
	\[
	\varliminf_{l\to +\infty} i_1(\gamma(t),[0,l])/l =\lim_{k\to +\infty} a_{q_k}.
	\]
	Then we get that 
	$\mIndex_U(B)=\varlimsup_{k\to +\infty}a_k$ and $\mIndex_L(B)=\varliminf_{k\to +\infty}a_k$.
	
	Then we only need to show that $\lim_{k\to +\infty} |a_{k+1}-a_k|=0$. It is a direct consequence of equation \eqref{eq:ind_dif}.  The Lemma then follows.
	
\end{proof}

We will show that  if $\mIndex_L(B)$ and $\mIndex_U(B)$ are considered as functionals of $B$,  then they are both uniformly continuous on $E_K$. 
We need some lemmas to prove it.

We define the functions  
\begin{align*}
	&f(B,n):=\MasSymp(\Id, \gamma_B(t),[0,n])\\
	&g(B,n):=\MasSymp(\gamma(n)), \gamma_B(t),[0,n])\\
	& h(B,n):=\hat i(\gamma_B(t),[0,n]),
\end{align*}
and let  $\mathcal{S}$  be the shift operator $\mathcal{S}B(t) =B(t+1)$.

Recall that for any path of symplectic matrices $\gamma(t)$,  we have
\begin{align} \label{eq:index_compare}
	\MasSymp(\gamma(b)),\gamma(t),[a,b])\le\MasSymp(\Lambda, \gamma(t),[a,b])\le \MasSymp(\gamma(b),\gamma(t),[a,b])\\
	\MasSymp(\gamma(b),\gamma(t),[a,b] \ge \MasSymp(\gamma(a),\gamma(t),[a,b])-2d.
\end{align}

\begin{lem}\label{lem:index_compare}
	We have 
	\begin{align*}
		f(B,n)\ge h(B,n)\ge g(B,n)\ge f(B,n)-2d, \\
		f(B,n+m)\le f(B,n)+f(\mathcal{S}^n B,m),\\
		g(B,n+m)\ge g(B,n)+g(\mathcal{S}^n B,m).
	\end{align*}
\end{lem}

\begin{proof}
	We have the formula 
	$\hat i(\gamma_B(t),[0,n])=\frac{1}{2\pi}\int_0^{2\pi}i_{e^{i\theta}}(\gamma_B(t),[0,n])d\theta$.
	By  \eqref{eq:index_compare}, we have  
	\[
	f(B,n)\ge i_{e^{i\theta}}(\gamma_B(t),[0,n]) \ge g(B,n).
	\]
	It follows that $ f(B,n)\ge \hat i(\gamma_B(t),[0,n]) =h(B,n)\ge g(B,n)$.
	
	By path additivity of Maslov index, we have
	\[f(B,n+m)=\MasSymp(\Id,\gamma_B,[0,n])+\MasSymp(\Id, \gamma_B,[n,n+m])\]
	Use symplectic invariance of Maslov index, we have 
	\begin{align*}
		\MasSymp(\Id,\gamma_B,[n,n+m])=\MasSymp( \gamma_B(n)^{-1}, \gamma_B(t)\gamma_B(n)^{-1},[n,n+m])\\
		=\MasSymp( \gamma_B(n)^{-1}, \gamma_{\mathcal{S}^n B},[0,m])\le \MasSymp( \Id,  \gamma_{\mathcal{S}^n B},[0,m])\\
		= f(\mathcal{S}^n B,m).
	\end{align*}
	Then we get  $ f(B,n+m)\le f(B,n)+f(\mathcal{S}^n B,m)$.
	Similarly, we have  $g(B,n+m)\ge g(B,n)+g(\mathcal{S}^n B,m)$.
\end{proof}

Let 
\begin{align*}
	F_{k,n}(B)=\sum_{l=0}^{n-1}f(\mathcal{S}^{2^k l}B,2^k)/n,\quad  F_k (B)=\varlimsup_{n\to \infty}F_{k,n}(B),\quad \underline F_k (B)=\varliminf_{n\to \infty}F_{k,n}(B),\\
	G_{k,n}(B)=\sum_{l=0}^{n-1}g(\mathcal{S}^{2^k l}B,2^k)/n, \quad G_k (B)=\varlimsup_{n\to \infty}G_{k,n}(B),\quad \underline G_k (B)=\varliminf_{n\to \infty}G_{k,n}(B),\\
	H_{k,n}(B)=\sum_{l=0}^{n-1}h(\mathcal{S}^{2^k l}B,2^k)/n, \quad H_k (B)=\varlimsup_{n\to \infty}H_{k,n}(B),\quad\underline H_k (B)=\varliminf_{n\to \infty}H_{k,n}(B).
\end{align*}

Then we have a formula to calculate $\mIndex_U(B)$.

\begin{lem}\label{lem:mean_index_formula}
	The limits of $H_k(B)$  and $\underline H_k(B)$ exist.  
	We have
	\[
	\mIndex_U(B)=\lim_{k\to +\infty} H_k(B)/2^k,  \quad \mIndex_L(B)=\lim_{k\to +\infty} \underline H_k(B)/2^k .
	\]	
\end{lem}
\begin{proof}
	We only prove the first equation, since the proof	of the other one is similar.
	
	By Lemma \ref{lem:index_compare},  we have
	\begin{equation}\label{eq:FKNHKN_compare}
		F_{k,n}(B)\ge H_{k,n}(B)\ge G_{k,n}(B)\ge  F_{k,n}(B)-2d .
	\end{equation}
	It follows that
	\begin{equation}\label{eq:FKHK_compare} 
		F_k(B)/2^k \ge H_k(B)/2^k\ge G_k(B)/2^k \ge F_k(B)/2^k -2d/2^k.
	\end{equation}
	Also	by Lemma \ref{lem:index_compare}, we have
	\[
	f(\mathcal{S}^{2^k l}B,2^k) \le   f(\mathcal{S}^{2^k l}B, 2^{k-1})+ f(\mathcal{S}^{2^k l+2^{k-1}}B, 2^{k-1}).
	\]
	It follows that 
	\[
	F_{k,n}(B)=\sum_{l=0}^{n-1}f(\mathcal{S}^{2^k l}B,2^k)/n \le  \sum_{l=0}^{2n-1}f(\mathcal{S}^{2^{k-1} l}B,2^{k-1})/n =2 F_{k-1,2n}(B) .
	\]
	
	Then we can conclude that   $F_{k-1,2n}(B)/2^{k-1} \ge   F_{k,n}(B)/2^k $. 
	It follows that 
	\[
	\varlimsup_{n\to +\infty} F_{k-1,n}(B)/2^{k-1}\ge \varlimsup_{n\to +\infty} F_{k-1,2n}(B)/2^{k-1}\ge  \varlimsup_{n\to +\infty} F_{k,n}(B)/2^k .
	\]
	Then we get $F_{k-1}(B)/2^{k-1} \ge F_k(B)/2^k$ . So  $\lim_{k\to +\infty} F_{k}(B)/2^k$ exists.
	By \eqref{eq:FKHK_compare}, we have  
	\[
	\lim_k H_k(B)/2^k=\lim_k F_k(B)/2^k=\lim_k G_k(B)/2^k .
	\]
	Now we will show that  $\mIndex_U(B)= \lim_k H_k(B)/2^k$ .
	By Lemma \ref{lem:index_compare}, we have 
	\begin{equation}\label{eq:index_compare2}
		G_{k,n} (B) \le g(B,n2^k)/n \le f(B, n2^k)/n \le F_{k,n}(B).
	\end{equation}
	Similar with Lemma \ref{lem:index_interval}, for any fixed $k$ , we have $ \mIndex_U(B)=\varlimsup_{n\to +\infty} f(B,n2^k)/(n2^k)$.
	By \eqref{eq:index_compare2}, we have 
	\begin{equation}\label{eq:upper_mean_index_bound}
		F_k(B)/2^k= \varlimsup_n F_{k,n}(B)/2^k \ge  \mIndex_U(B) \ge  \varlimsup_n G_{k,n}(x)/2^k =G_k(B)/2^k.
	\end{equation}
	Take limit, then we get  $\lim_k H_k(B)/2^k =\mIndex_U(B)$.
\end{proof}


{\bf Proof of Theorem \ref{thm:uniform conti}.}
\begin{proof}
	We only prove the continuity of $\mIndex_U(B)$ and $\mIndex_L(B)$. 	
	
	Step 1.
	
	The first step is to show the continuity of $h(B,1)$ on $E_K$.
	Let $\gamma_1=\gamma_{B_1},\gamma_2=\gamma_{B_2}$.
	Denote the matrix norm by $\|\cdot\|$.
	By using \cite[Chapter 3, 1.11]{Hale80},we have
	\begin{equation}\label{eq:compare_fund_sol}
		\|\gamma_1(t)-\gamma_2(t)\| \le \sup_{0\le s\le t}\|\gamma_1(t)\|(\exp\int_0^t \|B_2(s)\|ds)\times\int_0^t\|B_1(s)-B_2(s)\|ds
	\end{equation}
	and it follows that 
	$$
	\|\gamma_1(t)-\Id\|\le (\exp\int_0^t\|B_1(s)\|ds) \int_0^t\|B_1(s)\|ds. 
	$$
	Then we get 
	\begin{equation} \label{eq:solution_bound}
		\|\gamma_1(t)\|\le 1+ K e^K. 
	\end{equation} 
	Substitute it to \eqref{eq:compare_fund_sol}, then we see that 
	there is a constant   $C(K)$  such that   $\|\gamma_1(t)-\gamma_2(t)\|\le C(K ) \|B_1-B_2\|_{C^0} $  for  $t\in [0,n]$.
	
	Let $B_s= (1-s)B_1+sB_2$. Let $\gamma_{1+s}(t)$ be the associated fundamental matrix solution.

	Then by  homotopy invariance of Maslov index, we have
	\begin{equation}
		i_{e^{i\theta}}(\gamma_2,[0,1])-i_{e^{i\theta}}(\gamma_1,[0,1]) =i_{e^{i\theta}}(\gamma_s(1),s\in [0,1]) .
	\end{equation}
	By the definition of $\hat i$ , we have
	\begin{equation}\label{eq:dif_mean_index}
		\hat i(\gamma_2,[0,1])-\hat i(\gamma_1,[0,1])=\frac{1}{2\pi}\int_0^{2\pi}i_{e^{i\theta}}(\gamma_s(1),s\in [0,1])d\theta. 
	\end{equation}
	Since $B_s\in E$, like \eqref{eq:solution_bound}, we have $\|\gamma_s(1)\|\le 1+Ke^K$.
	Note that  the set of eigenvalues of matrix $M$ is continuous as a function of $M$.
	Then it is also uniformly continuous on $\set{M\in R^{2d\times 2d}|\|M\|<1+Ke^k}$.
	By \eqref{eq:compare_fund_sol},  $\|\gamma_s(1)-\gamma_1(1)\|\le C(K)s\|B_1-B_2\|_{C^0}$. So for any $\epsilon>0$,
	there is $\delta>0$ such that if $\|B_1-B_2\|_{C^0}<\delta$,  the measure of the set $F:=\set{\theta\in [0,2\pi]|e^{i\theta} \in \sigma(\gamma_s(1)) \mathrm{for\ some\ }  s\in [0,1]}$ is less than $\epsilon$.
	
	Then we can conclude that 
	\begin{equation}
		\left|\frac{1}{2\pi}\int_0^{2\pi}i_{e^{i\theta}}(\gamma_s(1),s\in [0,1])d\theta \right| \le \left|\frac{1}{2\pi}\int_{\theta\in F}i_{e^{i\theta}}(\gamma_s(1),s\in [0,1])d\theta  \right| .
	\end{equation}  Since $\gamma_s(1)$ in a small neighbourhood of $\gamma_1(1)$,   we have 
	$i_{e^{i\theta}}(\gamma_s(1),s\in [0,1])\le 2d$. It follows that 
	
	\begin{equation}
		\left|\frac{1}{2\pi}\int_0^{2\pi}i_{e^{i\theta}}(\gamma_s(1),s\in [0,1])d\theta \right| \le d\epsilon/\pi.
\end{equation}

	Then by \eqref{eq:dif_mean_index}, for any $B_1,B_2\in E$, if $\|B_1-B_2\|_{C^0}<\delta$,we have 
	\[
	|\hat i(\gamma_2,[0,1])-\hat i(\gamma_1,[0,1])|\le d\epsilon/\pi .
	\]
	Then we get the uniform continuity for $h(B,1)$ . Similarly, we also get the uniform continuity for $h(B,m)$ for any integer $m$.
	
	Step 2.
	
	By \eqref{eq:FKHK_compare} and \eqref{eq:upper_mean_index_bound}, we have
	\begin{align*}
		F_k(B)/2^k\ge \mIndex_U(B)\ge G_k(B)/2^k\ge F_k(B)/2^k-2d/2^k \\
		F_k(B)/2^k\ge H_k(B)\ge G_k(B)/2^k\ge F_k(B)/2^k-2d/2^k .
	\end{align*}
	It follows that $|\mIndex_U(B)-H_k(B)|\le 2d/2^k$.
	Then for each $\epsilon>0$, there is $k\in \N$ such that $|\mIndex_U(B)-H_k(B)|\le \epsilon$. Note that this $k$ is independent with $B$.
	
	Since $h(B,2^k)$ is uniformly continuous on $E$, for $B_1,B_2\in E$, there is $\delta>0$ such that if $\|B_1-B_2\|_{C^0}<\delta$  , then $|h(B_1,2^k)-h(B_2,2^k)|<\epsilon$.
	
	Note that $\|TB_1-TB_2\|_{C^0}=\|B_1-B_2\|_{C^0}$. Then we have
	\[
	|H_{k,n}(B_1)-H_{k,n}(B_2)|\le \sum_{l=0}^{n-1}|h(T^{2^k l}B_1,2^k)-h(T^{2^k l}B_2,2^k)|/n \le \epsilon.
	\]
	It follows that 
	\[
	|H_k(B_1)-H_k(B_2)|=|\varlimsup_n H_{k,n}(B_1) - \varlimsup_n H_{k,n}(B_2)|\le \epsilon
	\]
	Then we have $|H_k(B_1)/2^k-H_k(B_2)/2^k|\le \epsilon/2^k$. 
	Finally we can conclude that 
	\[
	|\mIndex_U(B_1)-\mIndex_U(B_2)|\le |\mIndex_U(B_1)-H_k(B_1)|+|\mIndex_U(B_2)-H_k(B_2)|+|H_k(B_1)-H_k(B_2)|\le 2\epsilon +\epsilon/2^k
	\]
	The theorem then follows.
	
\end{proof}

Using the continuity of $\mIndex(B)$  we have some asymptotic result.

\begin{cor} \label{cor:asymptotic}
	Assume that $\tilde B \in E_K$ with $\lim_{t\to +\infty} (\tilde B(t)-B(t))=0$.
	Then $\mIndex (B)=\mIndex(\tilde B)$.	
\end{cor}

\begin{proof}
	
	Let $B_s(t)=B(s+t)$ , $\tilde B_s(t)=\tilde B(s+t)$.  We have 
	\[
	\mIndex_{L}(B)=\mIndex_{L}(B_s)
	\]
	By Theorem \ref{thm:uniform conti}, for each $\epsilon >0$ , there is $\delta >0$ such that if $|B_s(t)-\tilde B_s(t)|<\delta$ on $\R^+$,  we have  $|\mIndex_L(B_s)-\mIndex_L(\tilde B_s)|<\epsilon$ .
	So for $s$ large enough, we have  $|\mIndex_L(B_s)-\mIndex_L(\tilde B_s)|<\epsilon$ .
	By the corollary \ref{cor:trans_invariant}, we see that $\mIndex_L$ and $\mIndex_U$ are invariant under translation.
	It follows that $|\mIndex_L(B)-\mIndex_L(\tilde B)|<\epsilon$ for each $\epsilon >0$. Then we have $\mIndex_L(B)=\mIndex_L(\tilde B)$.
	Similarly, we have $\mIndex_U(B)=\mIndex_U(\tilde B)$
\end{proof}

\section{The mean index  of quasi-periodic system}
The upper and lower mean index are the same for many important orbits in  Hamiltonian system. 
For example, they are the same for periodic orbits. In this cases, $\mIndex(x)\in\R$ is a point. 
In this section, we will show that $\mIndex_U$ and $\mIndex_L$ are the same for quasi-periodic orbits.

Now we only consider the linear equation of  quasi-periodic orbit. 
Assume $B(t)=S(p+q t)$  with $S\in C^0(\T^{m}\to \Sym(2n,\R))$, $p,q\in\R^m$ and $q=(q_1,q_2,\cdots,q_m)\in \R^{m}$ with $q_1,q_2,q_3,\cdots,q_m$ are independent over the rationals.
Choose  $u>0$ such that $q_1,q_2,\cdots,q_m,1/u$ are independent over the rationals.
Then   $\sum_{k=1}^m  n_k u q_k \notin \Z$  if integers $\set{n_k}$ are not all zero. 
Let $P: \T^m\to \T^m$ be the map $t\to t+uq$.
Then for each nonzero integer $l$ , $P^l $ is  an irrational rotation on torus and it has a unique ergodic measure which is the Lebesgue measure on $\T^m$.

To study the mean index of quasi-periodic system, we need an ergodic theorem.
\begin{thm}\label{thm:uniquely_ergodic}(\cite[Theorem 6.19]{Walter82})
	Let $T$ be a continuous transformation of a compact metrisahle space X. Assume that $T$ is uniquely ergodic.
	Let $\mu$ be the unique ergodic measure. Let $f\in C(X)$.
	Then $\sum_{i=1}^n f(T^i x)/n$ converge uniformly to $\int_X fd\mu$.
\end{thm}

\begin{thm}\label{thm:quasi_periodic_index}
	For quasi-periodic system, $\mIndex_U(B)=\mIndex_L(B)=\mIndex_U^-(B)=\mIndex_L^-(B)$.
\end{thm}
\begin{proof}

	Let $p\in \T^m$,  $B_p(t)=S(p+qt)$. 
	For simplicity, we modify the notations used in Section \ref{sec:property_mean_index}.
	
	Let
	\begin{align*}
		&f(p,n):=\MasSymp(Gr(\Id),\Graph \gamma_{B_p}(t),[0,un])\\
		&g(p,n):=\MasSymp(Gr(\gamma(un)),\Graph \gamma_{B_p}(t),[0,un])\\
		& h(p,n):=\hat i(\gamma_{B_p}(t),[0,un]),
	\end{align*}
and
	\begin{align*}
		&F_{k,n}(p)=\sum_{l=0}^{n-1}f(P^{2^k l}p,2^k)/n,  F_k (p)=\varlimsup_{n\to \infty}F_{k,n}(p),\underline F_k (p)=\varliminf_{n\to \infty}F_{k,n}(p)\\
		&G_{k,n}(p)=\sum_{l=0}^{n-1}g(P^{2^k l}p,2^k)/n,G_k (p)=\varlimsup_{n\to \infty}G_{k,n}(p), \underline G_k (p)=\varliminf_{n\to \infty}G_{k,n}(p)\\
		&H_{k,n}(p)=\sum_{l=0}^{n-1}h(P^{2^k l}p,2^k)/n,H_k (p)=\varlimsup_{n\to \infty}H_{k,n}(p),\underline H_k (p)=\varliminf_{n\to \infty}H_{k,n}(p).
	\end{align*}

	Like Lemma \ref{lem:mean_index_formula}, we have 
	\begin{align*}
		\mIndex_U(B_p)=\lim_{k\to +\infty} H_k(p)/(u2^k),  \quad \mIndex_L(B_p)=\lim_{k\to +\infty} \underline H_k(p)/(u 2^k) .	
	\end{align*}
	So we only need to show that  $\lim_n H_{k,n}(p)$ exists. 
	
	To use Theorem \ref{thm:uniquely_ergodic}, we need to show  $h(p,n)$ is continuous on $\T^m$.
	Let $d$ be the distance on $\T^m$.  Since $\T^m$ is compact, then $S$ is uniformly continuous on $\T^m$.
	Note that  $d(p_1+qt,p_2+qt)=d(p_1,p_2)$. So for each $\epsilon>0$, there is $\delta>0$ such that if $d(p_1,p_2)<\delta$ then $\|S(p_1+qt)-S(p_2+qt)\|<\epsilon$.
	
	It follows that $\|B_{p_1}(t)-B_{p_2}(t)\|=\|S(p_1+qt)-S(p_2+qt)\|<\epsilon$  if $d(p_1,p_2)<\delta$.
	By step 1 in the proof of Theorem \ref{thm:uniform conti}, we see that  $h(p,n)$ is continuous on $\T^m$, and then $H_{k,n}(p)$ is continuous. 
	Then by Theorem \ref{thm:uniquely_ergodic},  $\lim_n H_{k,n}(p)$ uniformly converge to a constant.
	We have
	\[
	\lim_n H_{k,n}(p)=\int_{\T^m} \sum_{l=0}^{n-1} h(P^{2^kl}p,2^k)/n d\mu =\int_{\T^m} h(p,2^k) d\mu.
	\]
	
	Now we calculate $\mIndex_U^-(B)$ and $\mIndex_L^-(B)$.
	Let $h^-(p,n)=\MasSymp(\Id,  \gamma_{B_p}(t),[0,-un])$.
	Replace $P$ by $P^{-1}$, we  can define $H_{k,n}^-(p)$.
	With the same method,we get 
	\begin{equation}
		\mIndex_U^-(B)=\mIndex_L^-(B)=\lim_{k\to +\infty} \int_{\T^m}h^-(p,2^k)/(-u2^k) d\mu .
	\end{equation}
	
	By the path additivity  of Maslov index, we have  $h^-(p,n)=-\MasSymp(\Id, \gamma_{B_p}(t),[-un,0])$.
	It follows that $h^-(p,n)=-\MasSymp( \Id, \gamma_{B_{p-un}}(t),[0,un])=-h(p-un,n)$.
	Then we have
	\[
	\mIndex_U^-(B)=\mIndex_L^-(B)=-\lim_{k\to +\infty} \int_{\T^m}h(p-u2^k,2^k)/(-u2^k) d\mu=\lim_{k\to +\infty} \int_{\T^m}h(P^{-2^k}p,2^k)/(u2^k) d\mu .
	\]
	Since the measure $\mu$ is invariant under transform $P$, we have $\mIndex_L(B)=\mIndex_U(B)=\mIndex_{L}^-(B)=\mIndex_U^-(B)$.
\end{proof}

With this theorem we immediately get Theorem \ref{thm:quasi index pt}.

Furthermore, we have
\begin{cor}\label{cor:uniformly_converge}
For quasi-periodic system, $\mIndex(B)=\lim\limits_{k\to +\infty}  h(p,2^k)/(u2^k)$ and 	$\lim\limits_{k\to +\infty}  h(p,2^k)/(u 2^k)$ converge uniformly for $p$.
\end{cor}
\begin{proof}
	Note that $\mIndex(B)=\lim_{k\to +\infty} H_k(p)/(u2^k)$ where $H_k(p)=\lim_{n\to +\infty} H_{k,n}(p)$ is independent with $p$. 
	Then for each $\epsilon >0$, there is $k_0$ such that  $|\mIndex(B)-H_k(p)/(u 2^k)|<\epsilon$ for each $k>k_0$.

Fix some $k>k_0$ such that  $2d/2^k <\epsilon$.  
By Theorem \ref{thm:quasi_periodic_index}, $\lim_{n\to +\infty} H_{k,n}(p)$ converge uniformly for $p$.
It follows that for each $\epsilon>0$ there is  $n_0$ for each $n>n_0$ ,  
\begin{equation}\label{eq:HKN_converge}
|H_k(p)-H_{k,n}(p)|<\epsilon.
\end{equation}
By Lemma \ref{lem:index_compare}, similar with \eqref{eq:FKNHKN_compare} \eqref{eq:index_compare2},  we have 
\begin{align}
F_{n,k}(p)\ge H_{k,n}(p) \ge G_{k,n}(p)\ge F_{n,k}(p)-2d \\
G_{k,n}(p)\le g(p,n2^k)/n \le h(p,n2^k)/n\le f(p,n2^k)/n \le F_{k,n}(p)
\end{align}
It follows that 
\[
|h(p,n2^k)/n- H_{k,n}(p)|\le 2d .
\]
By \eqref{eq:HKN_converge}, we have $|h(p,n2^k)/n- H_k(p)|\le 2d+\epsilon$. It follows that
\[
|h(p,n2^k)-\mIndex(B)|\le |\mIndex(B)-H_k(p)/(u2^k)|+|h(p,n2^k)/(u n2^k)- H_k(p)/(u2^k)|\le 2d/2^k+\epsilon +\epsilon/2^k\le 3\epsilon.
\]
Choose some $l$ such that $2^l>n_0$.
Then for $m>l+k$, we have 
\[
|h(p,2^{m})/(u2^m)-\mIndex(B)|<3\epsilon .
\]
The corollary then follows.
\end{proof}

{\bf Proof of Corollary \ref{cor:limit set}.}
\begin{proof}
Let $\phi(t,p)$ be  the flow of Hamiltonian system $\dot x= J\mathcal{H}'(x)$.  
We can denote the Low and upper mean index of $\phi(t,p)$ by $\mIndex_L(p),\mIndex_U(p)$ respectively, and  denote $\Lambda_\omega(\phi(\cdot,p))$ by $\Lambda_\omega(p)$. Then we have
$\Lambda_\omega(p)=\Lambda_\omega(\phi(t_0,p))$ for each $t_0\in \R$.
By Corollary \ref{cor:trans_invariant}, we have $\mIndex(p)=\mIndex(\phi(t_0,p))$.
Let $p_1=\tilde{x}(0)$, $p_2=x_2(0)$.
Since $\tilde{x}$ is a quasi-periodic orbit, $\Lambda_\omega(p_1)$ is the invariant torus. We simply denote $M=\Lambda_\omega(p_1)$.  

Let $B_p= H''(\phi(\cdot,p))$, then we have $\mIndex(B_p)=\mIndex(\phi(\cdot,p))$. 
Denote $P: \R^{2d}\mapsto \R^{2d}$ be the map $p\mapsto \phi(u,p)$. 
We use notations in Theorem  \ref{thm:quasi_periodic_index} with such $P$.
Then by Lemma \ref{lem:mean_index_formula}, we have 
\[\mIndex_U(p)=\lim_{k\to +\infty} H_k(p)/(u2^k).\]

By Theorem \ref{thm:quasi_periodic_index},  $\mIndex(p)=\lim_k H_k(p)/(u2^k)$ is a constant for $p \in M $ and it converge uniformly on $M$.
Since $\Lambda_\omega(p_2)=M$, we have
  \[
  \lim_{t\to +\infty} d(\phi(t,p_2),M)= 0. 
  \] 
  Then there is a compact neighborhood  $W$ of $M$ such that $\phi(t,p_2)\in W$ for $t\geq 0$.
  
By Corollary \ref{cor:uniformly_converge}, for each $\epsilon>0$ there is $k_0$ for each $k>k_0$ , \begin{equation}\label{eq:uniformly_converge}
	|h(r_m',2^k)/(u2^k)-\mIndex(p_2)|<\epsilon.
\end{equation}
Choose some $k>k_0$.  
  Note that $\phi $ is continuous on $[0,2^k]\times W$.
  Let 
  So $H''(\phi(\cdot,\cdot))$ is continuous  on $[0,2^k]\times W$. Then it is uniformly continuous on $[0,2^k]\times M$ by the compactness of $W$.

Note that $h(p,2^k)=h(B_p, u 2^k )$ .
Then by the continuity of $h(B,2^k)$ for $B$,    we see that $h(p,2^k)$ is uniformly continuous on $W$.

  Then we can conclude that for each $\epsilon>0$,  there is $\delta>0$ such that  $|h(u,2^k)-h(v,2^k)|<\epsilon$ ,for $u,v\in W$ with $d(u,v)<\delta$.
  
For each $\delta>0$ there is  $m_0>0$ such that $d(\phi(t,p_2),M)<\delta$ for each $t>m_02^k$.

Let $r_m=P^{m 2^k}\phi( 0,p_2)=\phi(m2^k, p_2)$. Then for each $m>m_0$, there is $r_m'\in M$ such that 
$d(r_m,r_m')<\delta$.
It follows that  $|h(r_m,2^k)-h(r_m',2^k)|<\epsilon$.
Then by \eqref{eq:uniformly_converge}, we get 
\[
|h(r_m,2^k)/(u2^k)-\mIndex(p_1)|\le 2\epsilon , m>m_0.
\]
It follows that 
\begin{multline}
|H_{n,k}(p_2)/(u2^k)-\mIndex(p_1)|\le \sum_{m=1}^{n_0}|h(r_m,2^k)/(u2^k)-\mIndex(p_1)|/n +\sum_{m={n_0+1}}^{n}|h(r_m,2^k)/(u2^k)-\mIndex(p_1)|/n \\
\le 2\epsilon +\sum_{m=1}^{n_0}|h(r_m,2^k)/(u2^k)-\mIndex(p_1)|/n.
\end{multline}
Take upper limit for $ n$  , then we get $|H_k(p_2)/(u2^k)-\mIndex(p_1)|\le 2\epsilon$. 
Take limit for $k$, then we get $|\mIndex_U(p_2)-\mIndex(p_1)|\le 2\epsilon$ for each $\epsilon>0$.
Then we get  $\mIndex_U(p_2)=\mIndex(p_1)$. Similarly we get $\mIndex_L(p_2)=\mIndex(p_1)$.

\end{proof}

In \cite{JM82}, Johnson and Moser  define the rotation number for almost periodic  Schrodinger operator $\mathcal{L}=-\dfrac{d^2}{dt^2}+q(t)$.
We will generalize the definition of rotation number to  2-dimensional quasi-periodic system and prove that it is   proportional to the mean index defined in this paper.

\begin{defn}
	Let $B\in C(\R, \Sym(2,\R))$, $z=(u,v)^T$  be a nonzero solution of $\eqref{eq:general linear_system}$.
 Then  $t \to \arg (u+iv)$ is a map from $\R$ to the unit circle $\T$.
	Let $\theta_0 $ be an argument of $u(0)+iv(0)$.
		By the homotopy lifting property,  there is a unique continuous function $\theta: \R \to \R $ such that $\theta(t)$ is the 
	argument of $u(t)+iv(t)$ and  $\theta(0)=\theta_0$.
	We define the rotation number by  $$R(B)=\lim_{t\to +\infty } \theta(t)/t.$$
\end{defn}

\begin{thm} \label{thm:rotation_mean_index}
	For two dimensional  quasi-periodic system, $\pi \mIndex(B)= R(B)$.
\end{thm}

\begin{proof}
	Note that $z(t)=\gamma(t)z(0)$, where $\gamma(t)$ is the fundamental matrix solution.
	Use polar decomposition, we have $\gamma(t)=M(t)U(t)$ with positive definite  symplectic matrix  $M(t)$ and orthogonal symplectic matrix $U(t)$  for each $t$ and $M(t),U(t)$ is continuous for $t$.
	Note that 2-dimensional orthogonal symplectic matrix has the form $e^{J\theta}$ with $J=\begin{pmatrix}
		0 &-1\\ 1&0
	\end{pmatrix}$.  So  $t\to S(t)$ is a map from $\R $ to the unit circle.   Using homotopy lifting property, 
	there is a continuous function $\phi(t)$ such that $\phi(0)=0$ and $U(t)=e^{J\phi(t)}$.
	Let   $f_s(t)=M(st)U(t)z(0)/|M(st)U(t)z(0)|$. It is a map from $[0,1]\times \R$ to unit circle.
	Note that  $f_0(t)=U(t)z(0)/|z(0)|$. Let $\theta_0$ be an argument of $u(0)+iv(0)$
	Then $f_0(t)=(\cos(\theta_0+\phi(t)),\sin(\theta_0+\phi(t)))^T$.
	
	Use homotopy lifting property,  there is a continuous function $\theta(s,t)$ such that
	\[f_s(t)=(\cos(\theta(s,t)),\sin(\theta(s,t)))^T\] and   $\theta(0,t)=\theta_0+\phi(t)$.
	
	Then we can conclude that the rotation number is 
	\[
	R(B)=\lim_{t\to +\infty} (\theta(1,t)-\theta_0)/t=\lim_{t\to +\infty}(\theta(1,t)-\theta(0,t)+\phi(t))/t .
	\]
	For a fixed $t$,   $M(st)$ is a positive definite matrix  and $(Mv,v )>0$ for any $v\neq 0$ .So the angle between $Mv$ and $v$ is between $-\pi/2$ and $\pi/2$.  It follows that $|\theta(1,t)-\theta(0,t)| \le \pi/2$.
	Then we can conclude that
	\begin{equation}\label{eq:rotation_cal}
		R(B)=\lim_{t\to +\infty}\phi(t)/t.
	\end{equation}
	
	Now we calculate the mean index.
	Note that $\mIndex(B)=\lim_{l\to +\infty}= i_1(\gamma(t),[0,l])/l$ .
	By the homotopy invariance of Maslov index, we have 
	\[
	i_1(\gamma(t),[0,l])=i_1(M(t)U(t),[0,l])=i_1(U(t),[0,l])+i_1(M(sl)U(l),s\in 
	[0,1]).
	\]
	We also have
	\[
	\MasSymp(\Id ,  M(sl)U(l),[0,1])=\MasSymp( U(l)^{-1} , M(sl),[0,1]).
	\]
	Note that $|\MasSymp( U(l)^{-1} , M(sl),[0,1]) -i_\omega (M(sl),[0,1]) |\le 2$ for some $\omega\neq 1$ on unit circle.   
	Since  all the eigenvalues of $M(sl)$ are real,  we have $i_\omega (M(sl),[0,1])=0$.
	Finally we can conclude that 
	\[
	\lim_{l\to +\infty} i_1(\gamma(t),[0,l])/l=\lim_{l\to +\infty} i_1(U(t),[0,l])/l =\lim_{l\to +\infty}  2 \frac{\phi(l)}{2\pi} /l.
	\]
	By \eqref{eq:rotation_cal}, we have $\pi \mIndex(B)= R(B)$.
\end{proof}

\section{The relation with essential spectrum}

In this section we study the relation of mean index and essential spectrum.
First we need a  theorem from \cite{Palmer88}.

Let
\begin{align*} 
	&\mathcal{A}:=-J\dfrac{d}{dt}-B(t):W^{1,2}(\R,\R^{2d})\subset L^2(\R,\R^{2d})\to  L^2(\R,\R^{2d})\\
	&\mathcal{A}_+:=-J\dfrac{d}{dt}-B(t):W^{1,2}(\R^+,\R^{2d})\subset L^2(\R^+,\R^{2d})\to  L^2(\R^+,\R^{2d})\\
	&\mathcal{A}_-:=-J\dfrac{d}{dt}-B(t):W^{1,2}(\R^-,\R^{2d})\subset L^2(\R^-,\R^{2d})\to  L^2(\R^-,\R^{2d})
\end{align*}
\begin{thm}\label{thm:determin_Fredholm}
	Let  $\gamma(t)$ be the fundamental solution of \eqref{eq:general linear_system}.
	Then the operator  $\mathcal{A}_+  $ are  Fredholm  if and only if there is a projection $P:\R^{2d}\to \R^{2d}$ and $C,\beta>0$ such that the inequalities
	\begin{align}
		\left|\gamma(t) P \gamma^{-1}(s)\right| \leq C e^{-\beta(t-s)} & (s \leq t) \\
		\left |\gamma(t)(\Id-P) \gamma^{-1}(s) \right| \leq C e^{-\beta(s-t)} & (s \geq t).
	\end{align}
	hold on  $
	\R^+$  respectively.
\end{thm}

If $B$ is periodic with period $T$, we have a simple criterion to determine Fredholmness. 

\begin{lem} \label{lem:Fredholm}
	$\mathcal{A}, \mathcal{A}_+, \mathcal{A}_-$ are Fredholm if and only if  $\sigma(\gamma(T))\cap \U =\emptyset$ where $\U$ is the unit circle.  
\end{lem}
\begin{proof}

	Let $M=\gamma(T)$.
	Since $B(t)$ is periodic, we have $\gamma(t)=\gamma(t-KT)M^k$ for $t\in [kT,(k+1)T]$. 
	
	$(\Leftarrow)$ We assume that $\sigma(M)\cap \U =\emptyset$ . Then $\R^{2d}=V\oplus W$ such that 
	$V,W$ are invariant subspace of $M$ and  $|M|_V| < c$ , $ |M^{-1}|_W| <c$ for some $0<c<1$.
	Let $P$ be the projection to $V$ with $\ker P=W$.  Then $\Id-P$  is the projection to $W$ and $MP=PM$.
	
	Assume that  $t\in [kT, (k+1)T]$ ,$ s\in [lT,(l+1)T]$. 
	We have 
	\begin{multline}
		|\gamma(t)P\gamma^{-1}(s)|=|\gamma(t-kT)M^k P M^{-l}\gamma(s-lT)|=|\gamma(t-kT)M^{k-l} P \gamma(s-lT)\\
		\le c^{k-l}|\gamma(t-kT)||\gamma(s-lT)| 
	\end{multline}
	and 
	\begin{multline}
		|\gamma(t)(\Id -P)\gamma^{-1}(s)|=|\gamma(t-kT)M^k (\Id -P) M^{-l}\gamma(s-lT)|=|\gamma(t-kT)M^{k-l} (\Id -P) \gamma(s-lT)\\
		\le c^{l-k}|\gamma(t-kT)||\gamma(s-lT)|. 
	\end{multline}	
	Let $Q=\max\set{|\gamma(t)|,t\in [0,T]}$. We have 
	\begin{align}
		\left|\gamma(t) P \gamma^{-1}(s)\right| \leq Q^2 c^{k-l} \le Q^2 c^{t-s-2T}  &  (s \leq t) \\
		\left |\gamma(t)(\Id-P) \gamma^{-1}(s) \right| \leq Q^2 c^{l-k}\le  Q^2 c^{s-t-2T}  & (s \geq t).
	\end{align}
	By theorem \ref{thm:determin_Fredholm}, we see that $\mathcal{A}, \mathcal{A}_+, \mathcal{A}_-$ are all Fredholm operators.
	
	$(\Rightarrow)$ Assume that  $\mathcal{A}$ is Fredholm. For integer $k>0$, let $t=kT, s=0$ , then we have
	\[
	|M^kP|=|\gamma(kT)P|\le Ce^{-\beta kT}.
	\]
	Similarly we have $|(\Id -P)M^{-k}| \le Ce^{-\beta kT} $. Assume that there is $\omega\in \U$ such that $\omega \in \sigma(M)$. Then there is  $x\in \R^{2d}, |x|=1$ such that $Mx=\omega x$.
	It follows that  $|(\Id-P) x|=|(\Id -P)\omega^{-k}x|=|(\Id -P)M^{-k}x|\le C e^{-\beta kT}$.
	Let $k\to +\infty$, then we get $x=Px$.
	It follows that 
	\[
	|x|=|\omega^k x|=|M^kx|=|M^kPx|\le Ce^{-\beta kT}
	\]
	Let $k\to +\infty$, then we get $|x|=0$ . It is a contradiction. So  $\sigma(M)\cap \U=\emptyset$.
	The results for $\mathcal{A}_-, \mathcal{A}_+$ are similar.
\end{proof}

Now we can show the relation of  mean index and Fredholmness for periodic system. Recall that for periodic system,  $T \mIndex(B)=\hat i(B)$, where  
	\[
	\hat i (B)=\hat i(\gamma(t),[0,T])=\frac{1}{2\pi}\int_{0}^{2\pi}i_{e^{i\theta}}(\gamma(t),[0,T])d\theta. 
	\]
	So we can prove the theorem for $\hat i$ instead. 

\begin{thm}\label{thm:Fredholm_periodic}
	$\mathcal{A}, \mathcal{A}_-, \mathcal{A}_+$ are Fredholm if and only if  $\hat{I}(B+\lambda I)=\hat{I}(B)$ for $|\lambda|$ small enough. 
\end{thm}

\begin{proof}
		 Let $B_s(t)=B(t)+s\Id$ and  $\gamma_s$ be the associated fundamental matrix solution, then 
			$\hat i(B_s)=\frac{1}{2\pi}\int_{0}^{2\pi}i_{e^{i\theta}}(\gamma_s(t),[0,T])d\theta$.

	$(\Leftarrow)$  Assume that $\mathcal{A}$ is not Fredholm. By Lemma \ref{lem:Fredholm},  there is $\omega_0 \in \U\cap \sigma(M)$, where  $M=\gamma_0(T)$.
		Let $s_0>0$, 
	using homotopy invariance  of Maslov index, we have
	\begin{equation}
		\hat i(B_{s_0})-\hat i(B_{-s_0})=\frac{1}{2\pi}\int_{0}^{2\pi}(i_{e^{i\theta}}(\gamma_{s_0}(t),[0,T])-i_{e^{i\theta}}(\gamma_{-s_0}(t),[0,T]))d\theta.
	\end{equation}
		We will show that $\gamma_s(T)$ is a positive path for $s$, that is, $-J\frac{\partial\gamma_s(T)}{\partial s}\gamma_s^{-1}(T)>0$.  Direct compute show that 
		\[  -\gamma^T_s(T)J \frac{\partial\gamma_s(T)}{\partial s}=\int_0^T\gamma^T_s(t) \gamma_s(t) dt>0,       \] which implies the result.  Then we have 
		\[\Delta(\omega):= i_\omega(\gamma_s(T), s\in[-s_0, s_0])=\sum_{\xi\in [-s_0,s_0)}\dim\ker (\gamma_\xi(T)-\omega I)\geq0, \quad \forall\omega\in\U \]

	By the homotopy invariance of Maslov index and spectral flow, we have 
	\[
	\Delta(\omega)=i_\omega(\gamma_{s_0}(t),t\in [0,T])-i_\omega(\gamma_{-s_0}(t),t\in [0,T])
	\]
	Since $\omega_0\in \U\cap \sigma(M)$ , we have $\dim \ker(M-\omega_0\Id)\neq 0$.
	Since $\gamma_s(T)$ is positive path, we have
	
	\begin{align}\label{eq:positive_sf}
		&\Delta(\omega_0)\ge \dim\ker (M-\omega_0\Id)>0,\\
		&\Delta(\omega)\ge \dim\ker ((M-\omega\Id))\ge 0 ,\forall\omega\in \U,\\
		&\ker(\gamma_{\pm s_0}(T)-\omega_0\Id)=0, \mathrm{for\ } s_0  \mathrm{\ small \ enough}. 
	\end{align}
			Then  $\omega_0$ has a neighborhood $V$ on $\U$, such that 
	\begin{equation}\label{eq:nondegenrate}
		\ker(\gamma_{\pm s_0}(T)-\omega \Id)=0, \quad \forall \omega\in V.
	\end{equation}
	By homotopy invariance of Maslov index, we have
	\[ i_\omega(\gamma_s(T),s\in [-s_0,s_0])=i_{\omega_0}(\gamma_s(T),s\in [-s_0,s_0])     \]
	
	
	Then for $\omega\in V$, we have $\Delta(\omega)=\Delta(\omega_0)>0$. 
	By \eqref{eq:positive_sf}, we have 
	\begin{multline}
		\hat i(B_{s_0})-\hat i(B_{-s_0})=\frac{1}{2\pi}\int_{0}^{2\pi}i_{e^{i\theta}}(\gamma_{s_0}(t),[0,T])-i_{e^{i\theta}}(\gamma_{-s_0}(t),[0,T]) d\theta=\frac{1}{2\pi}\int_0^{2\pi} \Delta(e^{i\theta}) d\theta\\
		\ge \frac{1}{2\pi}\int_{e^{i\theta}\in V} \Delta(e^{i\theta}) d\theta >0.
	\end{multline}
	So for any $s_0>0$, $\hat i(B+\lambda I)$ is not invariant for $\lambda\in [-s_0,s_0]$.
	
	$(\Rightarrow)$ Assume that $\mathcal{A}$ is Fredholm, then by Lemma \ref{lem:Fredholm}, $\sigma (\gamma(T))\cap \U =\emptyset$.
	It follows that  there is $s_0>0$ such that $\sigma (\gamma_{s}(T))\cap \U =\emptyset$ for $s\in [-s_0,s_0]$.
	Then we have 
	\[
	\Delta(\omega)=i_\omega\left(\gamma_s(T),s\in [-s_0,s_0]\right)=0, \forall \omega\in \U.
	\]
	It follows that 
	\[
	\hat i(B_{s_0})-\hat i(B_{-s_0})= \frac{1}{2\pi}\int_{e^{i\theta}\in \U} \Delta(e^{i\theta}) d\theta =0.
	\]
	By Lemma \ref{lem:monotone},  we see that $\hat i(B_s)$ is increasing for $s$. So $\hat i(B_s)=\hat i(B)$ for $s\in [-s_0,s_0]$.
	Similarly, the theorem also hold for $\mathcal{A}_-, \mathcal{A}_+$.
\end{proof}

Now we consider  perturbation of  periodic system.
\begin{cor}
	Assume that $B(t)$ is periodic with period $T$. 
	Assume that $\mathcal{A}$ is Fredholm.
	Then there is $\delta>0$ such that 	for $\tilde B \in C(\R, \R^{2d})$ if $|B(t)-\tilde B(t)|<\delta $ for $t\in \R$ then 
	$\mIndex_L(\tilde B)=\mIndex_U(\tilde B)=\mIndex(B)$  and  $\mIndex(\tilde B+\lambda \Id)=\mIndex(\tilde B)$  for $|\lambda|$ small enough. 
\end{cor}
\begin{proof}
	By Theorem \ref{thm:Fredholm_periodic}, there is $\delta>0$ such that $\mIndex (B+\lambda \Id)=\mIndex(B)$ for $\lambda\in [-\delta,\delta]$.  By Lemma \ref{lem:monotone}, we have 
	\[
	\mIndex(B+\delta \Id) \ge \mIndex_U(\tilde B)  \ge \mIndex_L(\tilde B)\ge \mIndex(B-\delta \Id) .
	\]
	It follows that $\mIndex_U(\tilde B)=\mIndex_L(\tilde B)=\mIndex(B)$.  
	For $|\lambda|$ small enough,  we have $B-\delta \Id <\tilde B+\lambda \Id <B+\delta\Id$. 
	Then similarly, we have $\mIndex(\tilde B+\lambda \Id)=\mIndex(B)$.
\end{proof}

We will give a non periodic example . It shows that  the upper and lower mean index may not the same for Fredholm operator.

\begin{ex} \label{ex:non pt}
	Let $\gamma(t)=e^{J\psi(t)}\begin{bmatrix}e^t &0\\ 0&e^{-t}\end{bmatrix},t \in \R^+$ with $\psi(t)=t \sin\frac{1}{t+1}$.
	
	Let  $B(t)=-J\dfrac{d}{dt} \gamma(t)\gamma(t)^{-1}$. We have 
	\begin{multline}
		-J\dfrac{d}{dt} \gamma(t)\gamma(t)^{-1}=-J J\dot\psi \begin{bmatrix}e^t &0\\ 0&e^{-t}\end{bmatrix}\begin{bmatrix}e^{-t} &0\\ 0&e^{t}\end{bmatrix}e^{-J\psi(t)} -J e^{J\psi(t)}\begin{bmatrix}e^t &0\\ 0&-e^{-t}\end{bmatrix}\begin{bmatrix}e^{-t} &0\\ 0&e^{t}\end{bmatrix}e^{-J\psi(t)}   \\
		=\left(\sin\frac{1}{t+1}-\frac{t}{(t+1)^2}\cos(t+1)\right)e^{-J\psi(t)}-J e^{J\psi(t)}\begin{bmatrix}1&0\\0&-1\end{bmatrix}e^{-J\psi(t)}
	\end{multline}
	Then $|B(t)| <3 $ for $t\in \R^+$, and $\gamma(t)$ is the fundamental solution of $-J\dfrac{d}{dt}z-B(t)z=0$.
	Similar with the method in theorem \ref{thm:rotation_mean_index}, we have 
	\[
	\mIndex_L(B)=\varliminf_{l\to +\infty} \psi(t)/t =-1 , \quad \mIndex_U(B)=\varlimsup_{l\to +\infty} \psi(t)/t =1.
	\]
	
\end{ex}

\section{Appendix}

In this section, we briefly review the index theory for symplectic path.  The detail  could be  found  in \cite{Lon02,LZ00a,LZ00b,HS09}

Following \cite{Lon02}, we define the following  hypersurface of codimension one in $\Sp(2n)$:
\[    \Sp(2n)_\omega^0=\{M\in\Sp(2n) | \det(M-\omega I)=0.    \]
For $M\in \Sp(2n)_\omega^0$, we define a co-orientation of  $\Sp(2n)_\omega^0$
at $M$ by the positive direction $\frac{d}{dt}Me^{tJ}|_{t=0}$ of the path $Me^{tJ}$ with $|t|$ sufficiently small. 
\begin{defn}\label{def:Maslov-type index} For $\omega\in \mathbb{U}$, $\gamma\in C([0,T], \Sp(2n))$ with $\gamma(0)=I$, we define
\[     i_\omega(\gamma)=[e^{-\epsilon J}\gamma: \Sp(2n)_\omega^0]+ \frac{1}{2} \dim\ker(I-\omega I),  \]
for $\epsilon$ small enough, where $[, : ,]$ is the intersection number.
\end{defn}
The mean index of periodic system on $[0,T]$ can be defined as
\begin{equation}\label{eq:def_mean_index_periodic}
	\hat i(\gamma,[0,T])=\lim_{n\to\infty}\frac{i_1(\gamma, [0,nT])}{n} =\frac{1}{2\pi}\int_0^{2\pi} i_{e^{i\theta}} (\gamma,[0,T])d\theta.
\end{equation}

Since we will use it  in this paper,  we give a small generalization of the standard Maslov index 
\begin{defn}\label{def:ge-Maslov-type index} For $M\in\Sp(2n)$,  $\gamma\in C([0,T], \Sp(2n))$, $\omega\in\U$,  we define
\[  \MasSymp(\omega M,\gamma)=[e^{-\epsilon J}\gamma M^{-1}: \Sp(2n)_\omega^0],  \]
for $\epsilon$ small enough.
\end{defn}
We use $\MasSymp$ instead $i$ to avoid misunderstanding.
Some property of the Maslov-type index can be found in \cite{LT15}. It is obvious that for $\gamma\in C([0,T], \Sp(2n))$ with $\gamma(0)=I$ \[ i_1(\gamma)+n= \MasSymp(I,\gamma). \] 
The Maslov-type index can be explained by the Maslov index theory. 
We  now briefly reviewing the  Maslov index theory
\cite{Arn67,CLM94, RS93}.    Let $(\mathbb{R}^{2n},\omega)$ be the standard
symplectic space  and $Lag(2n)$ the Lagrangian Grassmannian.
For two 
continuous paths $L_1(t), L_2(t)$, $t\in[a,b]$ in $Lag(2n)$,
the Maslov index $\MasSymp(L_1, L_2)$ is an integer .   
Here we use the definition from \cite{CLM94}.
We list several properties of the Maslov index. The details could be found
in \cite{CLM94}.

{\bf(Reparametrization invariance)\/}
Let
$\phi:[c,d]\rightarrow [a,b]$ be a continuous and piecewise smooth
function with $\phi(c)=a$, $\phi(d)=b$, then \begin{equation} \mu(L_1(t),
	L_2(t))=\mu(L_1(\phi(\tau)), L_2(\phi(\tau))). \label{adp1.1} \end{equation}

{\bf(Homotopy invariant with end points)\/}
For two continuous
families of Lagrangian path $L_1(s,t)$, $L_2(s,t)$, $0\leq s\leq 1$, $a\leq
t\leq b$ which satisfy that  $\dim L_1(s,a)\cap L_2(s,a)$ and $\dim
	L_1(s,b)\cap L_2(s,b)$ are constant, we have \begin{equation} \mu(L_1(0,t),
	L_2(0,t))=\mu(L_1(1,t),L_2(1,t)). \label{adp1.2} \end{equation}

{\bf(Path additivity)\/}
If $a<c<b$, then 
\begin{equation}\mu(L_1(t),L_2(t))=\mu(L_1(t),L_2(t)|_{[a,c]})+\mu(L_1(t),L_2(t)|_{[c,b]}).
	\label{adp1.3} \end{equation}

{\bf(Symplectic invariance)\/}
Let $\gamma(t)$, $t\in[a,b]$ be a
continuous path in $\Sp(2n)$, then 
\begin{equation} \mu(L_1(t),L_2(t))=\mu(\gamma(t)L_1(t), \gamma(t)L_2(t)). \label{adp1.4} \end{equation}

{\bf(Monotony property)\/}
Suppose for $j=1,2$, $L_j(t)=\gamma_j(t)V$, where $\dot{\gamma}_j(t)=JB_j(t)\gamma_j(t)$ with $\gamma_j(0)=I_{2n}$. If $B_1(t)\geq B_2(t)$, then for any 
$V_0, V\in Lag(2n)$, we have 
\begin{equation} \mu(V_0, \gamma_1V)\geq \mu(V_0, \gamma_2V).\label{adp1.6} \end{equation}


We have comparison results of Maslov index.

\begin{thm} [{\cite[Corollary  3.16 ]{ZWZ18}}]\label{thm:comparison}

		Let $\lambda\in C([a, b],  Lag(2n)$ be a Lagrangian path. Then for any $V_1,V_2 \in  Lag(2n )$, we have
	\begin{align*}
		& \mu(\lambda(b), \lambda) \le \mu (\mu_1, \lambda) \le \mu (\lambda(a), \lambda)\\
		& |\mu (V_1, \lambda)- \mu (V_2, \lambda)|\le n
	\end{align*}
\end{thm}

We will express the Maslov-type index by Maslov index. Please note that $(\R^{2n}\times\R^{2n}, -\Omega\times\Omega)$ is a $4n$-dimensional symplectic space. For $M\in \Sp(2n)$, let $Gr(M):=\{(x,Mx)|x\in\R^{2n}\}\in Lag(4n)$. We have 
\begin{equation} \MasSymp(M,\gamma)=\mu(Gr(M), Gr(\gamma)).
\end{equation}
Then the Maslov-type index has the same properties of Maslov index. From Theorem \ref{thm:comparison}, we have for $M_1,M\in\Sp(2n)$, $\gamma\in C([a,b], \Sp(2n))$, 
\begin{equation} \MasSymp(\gamma(b),\gamma)\leq\MasSymp(M,\gamma)\leq\MasSymp(\gamma(a),\gamma). \label{index comp}
\end{equation}
\begin{equation}\label{index fini d} |\MasSymp(M_1,\gamma)-\MasSymp(M,\gamma)|\leq 2n. 
\end{equation}

For $\gamma\in C^1([a,b], \Sp(2n))$ satisfied $\dot{\gamma}(t)=JB(t)\gamma(t)$. We say $\gamma$ is a positive path if $B(t)>0$ for $t\in[a,b]$. In this case, we have 
\begin{equation} \MasSymp(\omega M,\gamma)=\sum_{\xi\in[a,b)}\dim\ker(\gamma(\xi)-\omega M). \label{posi index}
\end{equation}

\begin{rem}
		
		In \cite{LZ00b}, the authors generalized the Maslov index to complex symplectic space.
		Let $(\cdot,\cdot)$ be the standard inner product of $\C^{2n}$.
		Define  $\omega(x,y)=(Jx,y)$ as the symplectic form on $\C^{2n}$.
		Then the Lagrangian Grassmannian can also be defined.
		Then the Maslov index can also be defined for a pair of continuous paths in complex Lagrangian Grassmannian and each property also holds for complex Maslov index.
		
		They also define the  complex symplectic group as $	\operatorname{Sp}(2 n, \mathbf{C})=\left\{M \in \operatorname{GL}(2 n, \mathbf{C}) \mid M^{*} J M=J\right\}$.
		Then  for real symplectic matrix $M$ and $\omega\in \U$, $\omega M$ is a complex symplectic matrix.
		
		So $\MasSymp(\omega M,\gamma)=\mu(\Gr(\omega M),\Gr(\gamma))$ is well defined and theorem \ref{thm:comparison} is also proved for complex Maslov index in  \cite{ZWZ18}.
\end{rem}

\section*{Acknowledgement}
 The first author sincerely thanks Yingfei Yi for the suggestion of using index theory  to study  the quasi-periodic orbits.

\vspace{1cm}
	\noindent
	\textsc{Prof. Xijun Hu}\\
	School of Mathematics\\
	Shandong University\\
	Jinan, Shandong, 250100 \\
	The People's Republic of China \\
	China\\
	E-mail: \email{xjhu@sdu.edu.cn}

\vspace{1cm}
\noindent
\textsc{Prof. Li Wu}\\
School of Mathematics\\
Shandong University\\
Jinan, Shandong, 250100\\
The People's Republic of China \\
China\\
E-mail:\email{vvvli@sdu.edu.cn}

\end{document}